\documentclass{amsart}
\usepackage{amsfonts, amsbsy, amsmath, amssymb}

\newtheorem{thm}{Theorem}[section]

\newtheorem{prop}[thm]{Proposition}

\newtheorem{rmk}[thm]{Remark}
\newtheorem{res}[thm]{Result}
\numberwithin{equation}{section}

\newcommand{\f}{\Bbb F}
\newcommand{\x}{{\tt X}}
\newcommand{\rr}{{\tt r}}
\newcommand{\z}{{\tt z}}

\begin{document}

\title[Permutation Polynomials of the form $\x^r(a+\x^{2(q-1)})$]{Permutation Polynomials of the form $\x^r(a+\x^{2(q-1)})$ --- A Nonexistence Result}

\author[Xiang-dong Hou]{Xiang-dong Hou}
\address{Department of Mathematics and Statistics,
University of South Florida, Tampa, FL 33620}
\email{xhou@usf.edu}


\subjclass[2000]{}

\begin{abstract}
Let $f=\x^r(a+\x^{2(q-1)})\in\f_{q^2}[\x]$, where $a\in\f_{q^2}^*$ and $r\ge 1$. The parameters $(q,r,a)$ for which $f$ is a permutation polynomial (PP) of $\f_{q^2}$ have been determined in the following cases: (i) $a^{q+1}=1$; (ii) $r=1$; (iii) $r=3$. These parameters together form three infinite families. For $r>3$ (there is a good reason not to consider $r=2$) and $a^{q+1}\ne 1$, computer search suggested that $f$ is not a PP of $\f_{q^2}$ when $q$ is not too small relative to $r$. In the present paper, we prove that this claim is true. In particular, for each $r>3$, there are only finitely many $(q,a)$, where $a^{q+1}\ne 1$, for which $f$ is a PP of $\f_{q^2}$. 
\end{abstract}

\maketitle

\section{Introduction}

A polynomial $f\in\f_q[\x]$ is called a permutation polynomial (PP) of $\f_q$ if it induces a permutation of $\f_q$.  Permutation binomials over finite fields in general are far from being well understood \cite{Hou-FFA-2015, Hou-CM-2015}. However, significant progress has been made towards understanding the permutation properties of more specific types of binomials. In this paper, we are interested in the binomials over $\f_{q^2}$ of the form
\[
f_{q,r,t,a}=\x^r(a+\x^{t(q-1)}),
\] 
where $1\le r\le q^2-2$, $1\le t\le q$, $a\in\f_{q^2}^*$, as PPs of $\f_{q^2}$. Such binomials were investigated in several recent papers \cite{Zieve-arxiv-1310.0776, Hou-FFA-2015,Hou-Lappano-JNT-2015,Lappano-FFA-2015, Hou-Fq12, Lappano-Fq12}. The results in these references are summarized as follows.

\begin{res}\label{R1.1}\cite[Corollary~5.3]{Zieve-arxiv-1310.0776} When $a^{q+1}=1$, $f_{q,r,t,a}$ is a PP of $\f_{q^2}$ if and only if $\text{\rm gcd}(r,q-1)=1$, $\text{\rm gcd}(r-t,q+1)=1$, and $(-a)^{(q+1)/\text{\rm gcd}(q+1,t)}\ne 1$.
\end{res}

\begin{res}\label{R1.2}\cite{Hou-FFA-2015} $f_{q,1,2,a}$ is a PP of $\f_{q^2}$ if and only if $q$ is odd and $(-a)^{(q+1)/2}=-1$ or $3$. (Note that $(-a)^{(q+1)/3}=3$ implies that $\text{\rm char}\,\f_q\ne 3$.)
\end{res}

\begin{res}\label{R1.3}\cite{Hou-Lappano-JNT-2015} $f_{q,1,3,a}$ is a PP of $\f_{q^2}$ if and only if one of the following occurs: (i) $q=2^e$, $e$ odd, $a^{q+1}=1$, and $a^{(q+1)/3}\ne 1$. (ii) $(q,a)$ belongs to a finite set which is determined in \cite{Hou-Lappano-JNT-2015}.
\end{res}

\begin{res}\label{R1.4}\cite{Lappano-FFA-2015} For $q\ge 5$, $f_{q,1,5,a}$ is a PP of $\f_{q^2}$ if and only if one of the following occurs: (i) $q=2^{4k+2}$ and $a^{(q+1)/5}\ne 1$ is a $5$th root of unity . (ii) $(q,a)$ belongs to a finite set which is determined in \cite{Lappano-FFA-2015}.
\end{res}

\begin{res}\label{R1.5}\cite{Lappano-FFA-2015} For $q\ge 7$, $f_{q,1,7,a}$ is a PP of $\f_{q^2}$ if and only if $(q,a)$ belongs to a finite set which is determined in \cite{Lappano-FFA-2015}.
\end{res}

\begin{res}\label{R1.6}\cite{Hou-Fq12} Let $t>2$ be a fixed prime. Under the assumption that $a^{q+1}\ne 1$, there are only finitely many $(q,a)$ for which $f_{q,1,t,a}$ is a PP of $\f_{q^2}$. 
\end{res}

\begin{res}\label{R1.7}\cite{Lappano-Fq12} $f_{q,3,2,a}$ is a PP of $\f_{q^2}$ if and only if $q$ is odd, $q\not\equiv 1\pmod 3$, and $(-a)^{(q+1)/2}=-1$ or $1/3$. (Note that $(-a)^{(q+1)/2}=1/3$ implies that $\text{\rm char}\,\f_q\ne 3$.) 
\end{res}

The PPs in Results~\ref{R1.2} and \ref{R1.7} under the condition $(-a)^{(q+1)/2}=-1$ and those in Results~\ref{R1.3} -- \ref{R1.5} under the condition $a^{q+1}=1$ are covered by Result~\ref{R1.1}. When $a^{q+1}=1$, the polynomial $f_{q,r,t,a}$ behaves nicely on $\f_{q^2}$ as a piecewise defined function, which is the reason behind Result~\ref{R1.1}. The PPs $f_{q,1,2,a}$ in Result~\ref{R1.2} under the condition $(-a)^{(q+1)/2}=3$ and $f_{q,3,2,a}$ in Result~\ref{R1.7} under the condition $(-a)^{(q+1)/2}=1/3$ owe their existence to more subtle reasons. 

The class $f_{q,r,1,a}$, that is, $t=1$, appears to have been overlooked, at least in the literature. However, this is a relatively easy case, and later in Theorem~\ref{T4.2} of the present paper, we will prove the following: $f_{q,r,1,a}$ is a PP of $\f_{q^2}$ if and only if $\text{gcd}(r,q-1)=1$, $q+1\mid r-1$, and $a^{q+1}\ne 1$.

A necessary condition for $f_{q,r,t,a}$ to be a PP of $\f_{q^2}$ is that $\text{gcd}(r,q-1)=1$. If $p=\text{char}\,\f_q$ divides $t$, then $f_{q,r,t,a}(\x)\equiv f_{q,r',t/p,a}(\x^p)\pmod{\x^{q^2}-\x}$, where $1\le r'\le q^2-2$ is such that $r'p\equiv r\pmod{q^2-1}$. Moreover, we have $f_{q,r,t,a}(\x)=f_{q,r/d,t/d,a}(\x^d)$, where $d=\text{gcd}(r,t)$. Therefore, we may assume that 
\begin{equation}\label{1.1}
\text{gcd}(rp,t(q-1))=1.
\end{equation}
Another necessary condition for $f_{q,r,t,a}$ to be a PP of $\f_{q^2}$ is that $(-a)^{(q+1)/\text{gcd}(q+1,t)}\ne 1$. (Otherwise, $f_{q,r,t,a}$ has at least two roots in $\f_{q^2}$.)

The aforementioned results allow us to make some observations about the class $f_{q,r,t,a}$ as a whole. Under the assumptions that $\text{gcd}(rp,t(q-1))=1$ and $(-a)^{(q+1)/\text{gcd}(q+1,t)}\ne 1$, there are four infinite families of parameters $(q,r,t,a)$ for which $f_{q,r,t,a}$ is a PP of $\f_{q^2}$:

\begin{itemize}
  \item [(i)] $a^{q+1}=1$, $\text{gcd}(r-t,q+1)=1$;
  \item [(ii)] $t=1$, $q+1\mid r-1$;
  \item [(iii)] $r=1$, $t=2$, $(-a)^{(q+1)/2}=3$;
  \item [(iv)] $r=3$, $t=2$, $(-a)^{(q+1)/2}=1/3$.
\end{itemize} 
These are probably the only infinite families. It is likely true that for each given $(r,t)$ with either $r>3$ or $t>2$, there are only finitely many $(q,a)$ with $a^{q+1}\ne 1$ and $\text{gcd}(rp,t(q-1))=1$ for which $f_{q,r,t,a}$ is a PP of $\f_{q^2}$. This claim has been confirmed by Result~\ref{R1.6} for $r=1$ and $t>2$. In the present paper, we confirm the same for $r>3$ and $t=2$. The precise statement of the theorem and an outline of its proof are given in the next section.

Additional works by several authors on the topic that are in progress have not been indicated in the present paper; interested readers may find them in the near future.

{\bf Acknowledgments.} The author would like to thank Qiang Wang and Stephan Lappano for the discussions that partially motivated the work in the present paper.

\section{Statement of the Theorem and Outline of the Proof}

\subsection{Statement of the theorem}\

Let $f=f_{q,r,2,a}$, where $a\in\f_{q^2}^*$, $a^{q+1}\ne 1$, and assume, as in \eqref{1.1}, that $\text{gcd}(rp,2(q-1))=1$, where $p=\text{char}\,\f_q$, that is, $r$ and $q$ are both odd and $\text{gcd}(r,q-1)=1$. We will show that if $r>3$ and $q$ is not too small relative to $r$, then $f$ is not a PP of $\f_{q^2}$. More precisely, our main result is the following theorem.

\begin{thm}\label{T2.1} 
Let $f=f_{q,r,2,a}=\x^r(a+\x^{2(q-1)})$, where $r$ and $q$ are both odd, $r>3$, and $a\in\f_{q^2}^*$ is such that $a^{q+1}\ne 1$. Then $f$ is not a PP of $\f_{q^2}$ if 
\[
q\ge
\begin{cases}
r^2-4r+5&\text{if}\ r\equiv 3\pmod p,\cr
8r-15&\text{if $r\not\equiv 3\pmod p$ and either $p=3$ or $r\equiv 7/4\pmod p$},\cr
6r-11&\text{if $p>3$ and $r\not\equiv 3,\ 7/4\pmod p$}.
\end{cases}
\]
\end{thm}

\subsection{Outline of the proof}\

Among the power sums $\sum_{x\in\f_{q^2}}f(x)^s$, where $1\le s\le q^2-2$, the useful ones are those with $s=\alpha+(q-1-\alpha)q$, where $\alpha$ is odd and $1\le \alpha\le q-2$; the others are automatically $0$. The sum $S(\alpha)=\sum_{x\in\f_{q^2}}f(x)^{\alpha+(q-1-\alpha)q}$ is computed and the result can be made explicit for small values of $\alpha$. Assume to the contrary that $f$ is a PP of $\f_{q^2}$. We then exploit the consequence that $S(\alpha)=0$ for all odd $\alpha$ with $1\le \alpha\le q-2$. The consideration of $S(1)=S(3)=S(5)=0$ produces a contradiction except for a few special cases: $p=3$ or $r\equiv 3,\ 3/2,\ 7/4\pmod p$. The case $r=3/2$ requires minimum effort. When $r\equiv 3\pmod p$, we examine an additional equation $S(p^l)=0$ for a suitable $l$ to reach a contradiction. When $r\not\equiv 3\pmod p$ and either $p=3$ or $r\equiv 7/4\pmod p$, useful information is extracted from the equation $S(7)=0$ to settle the case.

The sum $S(\alpha)$ can be expressed as a polynomial in $r$ and $z=(-a)^{-q(q+1)/2}$ with coefficients in $\Bbb Q$. (More precisely, $S(\alpha)=\sum_ia_i(r)z^i$, where $a_i\in\Bbb Q[\rr]$ is such that $a_i(\Bbb Z)\subset\Bbb Z$.) Our approach relies on computations of resultants of polynomials; such computations are easily performed with various symbolic computation programs.

\subsection{Outline of the paper}\

In Section~3, we compute the power sum $\sum_{x\in\f_{q^2}}f_{q,r,2,a}(x)^s$. Section~4 is a brief detour to the case $t=1$. The power sum $\sum_{x\in\f_{q^2}}f_{q,r,1,a}(x)^s$ is obtained by an easy adaptation of the computation in Section~3. The result allows us to determine the necessary and sufficient conditions on $(q,r,a)$ for $f_{q,r,1,a}$ to be a PP of $\f_{q^2}$. The remaining three sections constitute the proof of Theorem~\ref{T2.1}, in three cases: Section~5: $p>3$ and $r\not\equiv 3,\ 7/4\pmod p$; Section~6: $r\equiv 3\pmod p$; Section~7: $r\not\equiv 3\pmod p$ and either $p=3$ or $r\equiv 7/4\pmod p$.

Throughout the paper, letters in the typewriter font $\x,\rr,\z$ always denote indeterminates. If the primary use of a polynomial $A$ is its values $A(r)$ for a parameter $r$, then the indeterminate of $A$ is designated as $\rr$. The characteristic of $\f_q$ is always denoted by $p$.

\section{Power Sums}

Assume that $q$ is odd. Write $f=f_{q,r,2,a}=\x^r(a+\x^{2(q-1)})$, where $r\ge 1$, $\text{gcd}(r,q-1)=1$, and $a\in\f_{q^2}^*$. For $1\le s\le q^2-2$, written in the form $s=\alpha+\beta q$, $0\le \alpha,\beta\le q-1$, we have
\begin{equation}\label{3.1}
\begin{split}
\sum_{x\in\f_{a^2}}f(x)^s\,&=\sum_{x\in\f_{q^2}^*}x^{r(\alpha+\beta q)}(a+x^{2(q-1)})^{\alpha+\beta q}\cr
&=\sum_{x\in\f_{q^2}^*}x^{r(\alpha+\beta q)}\sum_{i,j}\binom\alpha i\binom\beta j x^{2(q-1)(i+jq)}a^{\alpha+\beta q-(i+jq)}\cr
&=\sum_{i,j}\binom\alpha i\binom\beta j a^{\alpha+\beta q-(i+jq)}\sum_{x\in\f_{q^2}^*}x^{r(\alpha+\beta q)+2(q-1)(i-j)}.
\end{split}
\end{equation}
The inner sum in the above is $0$ unless $\alpha+\beta q\equiv 0\pmod {q-1}$, i.e., $\alpha+\beta=q-1$.

Assume that $\alpha+\beta=q-1$. Since $\alpha+\beta q\equiv(\alpha+1)(1-q)\pmod{q^2-1}$, \eqref{3.1} becomes
\begin{equation}\label{3.2}
-\sum_{x\in\f_{a^2}}f(x)^s=\sum_{2(i-j)-(\alpha+1)r\equiv0\, \text{(mod\,$q+1$)}}\binom\alpha i\binom{q-1-\alpha}ja^{(\alpha+1)(1-q)-(i+jq)}.
\end{equation}
The above sum is $0$ unless $\alpha$ is odd. We assume that $\alpha$ is odd. Write
\begin{equation}\label{3.3}
(\alpha+1)r-2\alpha=c(q+1)-d,\quad 0\le d<q+1,\quad c=\left\lceil\frac{(\alpha+1)r-2\alpha}{q+1}\right\rceil.
\end{equation}
We claim that the conditions $0\le i\le \alpha$, $0\le j\le q-1-\alpha$ and $2(i-j)-(\alpha+1)r\equiv 0\pmod{q+1}$ together imply that $2(i-j)-(\alpha+1)r\in\{-c(q+1),-(c+1)(q+1)\}$. In fact, we have 
\[
2(i-j)-(\alpha+1)r\le 2\alpha-(\alpha+1)r=-c(q+1)+d<(-c+1)(q+1)
\]
and
\[
\begin{split}
2(i-j)-(\alpha+1)r\,&\ge -2(q-1-\alpha)-(\alpha+1)r=-2(q-1)-c(q+1)+d\cr
&=(-c-2)(q+1)+4+d>(-c-2)(q+1).
\end{split}
\]
Therefore \eqref{3.2} becomes
\begin{equation}\label{3.4}
\begin{split}
-\sum_{x\in\f_{a^2}}f(x)^s\,&=\sum_{2(i-j)-(\alpha+1)r=-c(q+1),-(c+1)(q+1)}\binom\alpha i\binom{q-1-\alpha}ja^{(\alpha+1)(1-q)-(i+jq)}\cr
&=\sum_{\substack{2(i-j)-(\alpha+1)r=-c(q+1),-(c+1)(q+1)\cr -\alpha\le j\le q-1}}\binom\alpha i\binom{\alpha+j}\alpha (-1)^ja^{(\alpha+1)(1-q)-(i+jq)}\cr
&=a^{(\alpha+1)(1-q)}\sum_i\binom\alpha ia^{-i}\sum_{\substack{j=i-\alpha+\frac d2,\,i-\alpha+\frac d2+\frac{q+1}2\cr j\le q-1}}\binom{\alpha+j}\alpha(-1)^ja^{-jq}.
\end{split}
\end{equation}
(In the second line of \eqref{3.4}, the condition $-\alpha\le j\le q-1$ is needed for the identity $\binom{q-1-\alpha}j=\binom{-1-\alpha}j=\binom{\alpha+j}\alpha(-1)^j$.) 

Recall from \eqref{3.3} that $d$ is even and $0\le d<q+1$. Therefore for $0\le i\le \alpha$, the inequality $i-\alpha+\frac d2+\frac{q+1}2\ge q$ holds if and only if $i=\alpha$ and $d=q-1$.

We first assume that $d\ne q-1$. Then \eqref{3.4} becomes 
\begin{equation}\label{3.5}
\begin{split}
&-\sum_{x\in\f_{q^2}}f(x)^s\cr
=\,&a^{(\alpha+1)(1-q)}\sum_i\binom\alpha i a^{-i}\biggl[\binom{i+\frac d2}\alpha(-1)^{i+\frac d2+1}a^{-q(i-\alpha+\frac d2)}\cr
&\kern 3.5cm +\binom{i+\frac d2+\frac 12}\alpha (-1)^{i+\frac d2+1+\frac{q+1}2}a^{-q(i-\alpha+\frac d2+\frac{q+1}2)}\biggr]\cr
=\,&(-1)^{\frac d2+1}a^{\alpha+1-q(1+\frac d2)}\sum_i \binom\alpha i(-1)^i\biggl[\binom{i+\frac d2}\alpha a^{-i(q+1)}\cr
&\kern 3.5cm +\binom{i+\frac d2+\frac 12}\alpha(-1)^{\frac{q+1}2}a^{-i(q+1)-\frac 12q(q+1)}\biggr].
\end{split}
\end{equation}
Setting
\begin{equation}\label{3.6}
z=(-a)^{-\frac 12q(q+1)},
\end{equation}
\eqref{3.5} becomes
\begin{equation}\label{3.7}
\sum_{x\in\f_{q^2}}f(x)^s=(-1)^{\frac d2}a^{\alpha+1-q(1+\frac d2)}\sum_i\binom\alpha i(-1)^i\biggl[\binom{i+\frac d2}\alpha z^{2i}+\binom{i+\frac d2+\frac 12}\alpha z^{2i+1}\biggr].
\end{equation}

Now assume that $d=q-1$. The computation from \eqref{3.4} to \eqref{3.7} holds after the term with $i=\alpha$ and $j=\frac d2+\frac{q+1}2=q$ is subtracted from the sum, that is 
\begin{equation}\label{3.8}
\begin{split}
\sum_{x\in\f_{q^2}}f(x)^s\,&=(-1)^{\frac d2}a^{\alpha+1-q(1+\frac d2)}\biggl[\sum_i\binom\alpha i(-1)^i\biggl(\binom{i-\frac 12}\alpha z^{2i}+\binom{i}\alpha z^{2i+1}\biggr)-z^{2\alpha+1}\biggr]\cr
&=-a^{\alpha+1}z\sum_i\binom\alpha i\binom{i-\frac 12}\alpha(-1)^iz^{2i}.
\end{split}
\end{equation}

To summarize, we have proved the following 
\begin{prop}\label{P3.1}
Let $q$ be odd and $\text{\rm gcd}(r,q-1)=1$. For $1\le s\le q^2-2$, written in the form $s=\alpha+\beta q$, where $0\le \alpha,\beta\le q-1$, we have

\begin{equation}\label{3.9}
\begin{split}
&\sum_{x\in\f_{q^2}}f(x)^s\cr
=\,&
\begin{cases}
\displaystyle (-1)^{\frac d2}a^{\alpha+1-q(1+\frac d2)}\sum_i\binom\alpha i(-1)^i\biggl[\binom{i+\frac d2}\alpha z^{2i}+\binom{i+\frac d2+\frac 12}\alpha z^{2i+1}\biggr]\cr
\kern5.95cm \text{if $\alpha$ is odd, $\alpha+\beta=q-1$, and $d\ne q-1$},\vspace{2mm}\cr
\displaystyle -a^{\alpha+1}z\sum_i\binom\alpha i\binom{i-\frac 12}\alpha(-1)^iz^{2i}\kern 0.8cm \text{if $\alpha$ is odd, $\alpha+\beta=q-1$, and $d=q-1$},\vspace{2mm}\cr
0\kern 5.75cm \text{otherwise},
\end{cases}
\end{split}
\end{equation}
where $d$ and $z$ are given in \eqref{3.3} and \eqref{3.6}, respectively.
\end{prop}

\begin{rmk}\label{R3.2}\rm
\begin{itemize}
  \item [(i)] In \eqref{3.3}, $d=q-1$ if and only if $(r-2)(\alpha+1)=(c-1)(q+1)$. The equation $(r-2)(\alpha+1)=(c-1)(q+1)$ has a solution $(\alpha,c)\in\Bbb Z^2$ with $1\le \alpha+1\le q$ if and only if $\text{gcd}(r-2,q+1)>1$.
\medskip  

  \item [(ii)] The PPs in Results~\ref{R1.2} and \ref{R1.7} corresponding to $(r,z)=(1,1/3)$ and $(3,3)$ are quite nontrivial. When $(r,z)=(1,1/3)$, we have $d=\alpha-1$. In this case, we know that for all odd $\alpha>0$, the identity
\begin{equation}\label{3.10}
\sum_i\binom\alpha i(-1)^i\biggl[\binom{i+\frac{\alpha-1}2}\alpha\Bigl(\frac 13\Bigr)^{2i}+\binom{i+\frac\alpha 2}\alpha\Bigl(\frac 13\Bigr)^{2i+1}\biggr]=0
\end{equation}
holds in $\Bbb Q$; see \cite{Hou-JNT-2013}. When $(r,z)=(3,3)$, $d=q-2-\alpha$. In this case, for all odd $\alpha>0$, we have in $\Bbb Q$ that 
\begin{equation}\label{3.11}
\sum_i\binom\alpha i(-1)^i\biggl[\binom{i-1-\frac{\alpha}2}\alpha 3^{2i}+\binom{i-\frac{\alpha+1}2}\alpha 3^{2i+1}\biggr]=0,
\end{equation}
which is equivalent \eqref{3.10}. The fact that $f_{q,1,2,a}$ (with $(-a)^{(q+1)/2}=3$) is a PP of $\f_{q^2}$ follows from \eqref{3.9} and \eqref{3.10}; the fact that $f_{q,3,2,a}$ (with $(-a)^{(q+1)/2}=1/3$) is a PP of $\f_{q^2}$ follows from \eqref{3.9} and \eqref{3.11}. Although \eqref{3.10} and \eqref{3.11} are easily seen to be equivalent, it is not clear how $f_{q,1,2,a}$ with $(-a)^{(q+1)/2}=3$ and $f_{q,3,2,a}$ with $(-a)^{(q+1)/2}=1/3$ are related.
\end{itemize}
\end{rmk}

\section{The Case $t=1$}

Let $g=f_{q,r,1,a}=\x^r(a+\x^{q-1})$, where $r\ge 1$, $\text{gcd}(r,q-1)=1$, and $a\in\f_{q^2}^*$. The power sum of $g$ can be easily obtained by an adaptation of the computation in Section~3. 

Let $1\le s\le q^2-2$ be given in the form $s=\alpha+\beta q$, $0\le \alpha,\beta\le q-1$. If $\alpha+\beta\ne q-1$, we have $\sum_{x\in\f_{q^2}}g(x)^s=0$. When $\alpha+\beta=q-1$, comparing with \eqref{3.2}, we have 
\begin{equation}\label{4.1}
-\sum_{x\in\f_{q^2}}g(x)^s=\sum_{i-j-(\alpha+1)r\equiv0\,\text{(mod\,$q+1$)}}\binom\alpha i\binom{q-1-\alpha}ja^{(\alpha+1)(1-q)-(i+jq)}.
\end{equation}
Write
\begin{equation}\label{4.2}
(\alpha+1)-\alpha=c'(q+1)+d',\quad 0\le d'<q+1,\quad c'=\left\lceil\frac{(\alpha+1)r-\alpha}{q+1}\right\rceil.
\end{equation}
The conditions $0\le i\le \alpha$, $0\le j\le q-1-\alpha$ and $i-j-(\alpha+1)r\equiv 0\pmod{q+1}$ together imply that $i-j-(\alpha+1)r=-c'(q+1)$. Therefore \eqref{4.1} becomes
\[
\begin{split}
-\sum_{x\in\f_{q^2}}g(x)^s\,&=\sum_{i-j-(\alpha+1)r=-c'(q+1)}\binom\alpha i\binom{q-1-\alpha}ja^{(\alpha+1)(1-q)-(i+jq)}\cr
&=\sum_{\substack{i-j-(\alpha+1)r=-c'(q+1)\cr -\alpha\le j\le q-1}}\binom\alpha i\binom{\alpha+j}\alpha (-1)^ja^{(\alpha+1)(1-q)-(i+jq)}\cr
&=a^{(\alpha+1)(1-q)}\sum_{i-\alpha+d'\le q-1}\binom\alpha i\binom{i+d'}\alpha(-1)^{i-\alpha+d'}a^{-i-(i-\alpha+d')q}\cr
&=(-1)^{\alpha+d'}a^{\alpha+1-q(1+d')}\sum_{i-\alpha+d'\le q-1}\binom\alpha i\binom{i+d'}\alpha(-1)^ia^{-i(1+q)}.
\end{split}
\]
When $d'<q$, we have
\[
-\sum_{x\in\f_{q^2}}g(x)^s=(-1)^{\alpha+d'}a^{\alpha+1-q(1+d')}\sum_i\binom\alpha i\binom{i+d'}\alpha(-1)^ia^{-i(1+q)}.
\]
When $d'=q$, 
\[
-\sum_{x\in\f_{q^2}}g(x)^s=(-1)^{\alpha+1}a^{\alpha-q}\sum_{i\le \alpha-1}\binom\alpha i\binom{i}\alpha(-1)^ia^{-i(1+q)}=0.
\]
To summarize, we have the following proposition.

\begin{prop}\label{P4.1}
Assume that $\text{\rm gcd}(r,q-1)=1$. For $1\le s\le q^2-2$ written in the form $s=\alpha+\beta q$, where $0\le \alpha,\beta\le q-1$, we have
\begin{equation}\label{4.3}
\sum_{x\in\f{q^2}}g(x)^s=
\begin{cases}
\displaystyle (-1)^{\alpha+d'+1}a^{\alpha+1-q(1+d')}\sum_i\binom\alpha i\binom{i+d'}\alpha(-1)^ia^{-i(1+q)}\cr
\kern 5cm \text{if $\alpha+\beta=q-1$ and $d'<q$},\vspace{2mm}\cr
0\kern 4.8cm\text{otherwise}.
\end{cases}
\end{equation}
\end{prop}

\begin{thm}\label{T4.2}
For $r\ge 1$ and $a\in\f_{q^2}^*$, $g=f_{q,r,1,a}$ is a PP of $\f_{q^2}$ if and only if $\text{\rm gcd}(r,q-1)=1$, $q+1\mid r-1$, and $a^{q+1}\ne 1$.
\end{thm}

\begin{proof}
($\Rightarrow$) We already know that the necessary conditions include $\text{gcd}(r,q-1)=1$ and $(-a)^{q+1}\ne 1$, i.e., $a^{q+1}\ne 1$. Assume to the contrary that $q+1\nmid r-1$. Choose $\alpha=0$. Then by \eqref{4.2}, $d'\ne q$, and hence \eqref{4.3} gives $\sum_{x\in\f_{q^2}}g(x)^{(q-1)q}\ne 0$, which is a contradiction.

($\Leftarrow$) Since $a^{q+1}\ne 1$, it follows that $0$ is the only root of $g$ in $\f_{q^2}$. Since $r\equiv 1\pmod{q+1}$, we have $d'=q$ in \eqref{4.2}. Thus by \eqref{4.3}, $\sum_{x\in\f_{q^2}}g(x)^s=0$ for all $1\le s\le q^2-2$. Hence $g$ is a PP of $\f_{q^2}$.
\end{proof}

\section{Proof of Theorem~\ref{T2.1} with $p>3$ and $r\not\equiv 3,\ 7/4\pmod p$}

Recall that $r$ and $q$ are both odd, $r>3$, and $a\in\f_{q^2}^*$ is such that $a^{q+1}\ne 1$. We assume that $q\ge 6r-11$, $p>3$, and $r\not\equiv 3,\ 7/4\pmod p$.

Assume to the contrary that $f=f_{q,r,2,a}$ is a PP of $\f_{q^2}$.

Since $q\ge 6r-11$, in \eqref{3.3} we have
\begin{equation}\label{5.1}
c=\left\lceil\frac{(\alpha+1)r-2\alpha}{q+1}\right\rceil=1 \quad\text{for}\ 1\le \alpha\le 5.
\end{equation}
By Remark~\ref{R3.2} (i), $d\ne q-1$ whenever $c=1$. For $\alpha=1,3,5$, \eqref{3.3} gives
\begin{equation}\label{5.2}
d=q+1+2\alpha-(\alpha+1)r.
\end{equation}
By \eqref{3.9} and \eqref{5.2}, for $\alpha=1,3,5$,
\begin{equation}\label{5.3}
\begin{split}
\Theta(\alpha):\,&=\sum_i\binom\alpha i(-1)^i\biggl[\binom{i+\frac 12+\alpha-\frac12(\alpha+1)r}\alpha z^{2i}+\binom{i+1+\alpha-\frac12(\alpha+1)r}\alpha z^{2i+1}\biggr]\cr
&=0,
\end{split}
\end{equation}
where $z$ is given in \eqref{3.6}. The expression $\Theta(\alpha)$ is a polynomial in $z$ with coefficients in $\Bbb Z[1/2]$. In fact,
\begin{align}
\label{5.4}
\Theta(1)&=\frac 12(1+z)A_1(r,z),\\
\label{5.5}
\Theta(3)&=\frac 1{2^4\cdot 3}(1+z)A_3(r,z),\\
\label{5.6}
\Theta(3)&=\frac 1{2^8\cdot 5}(1+z)A_5(r,z),
\end{align}
where
\begin{equation}\label{5.7}
A_1(\rr,\z)=(2 \rr-6) \z^2+\z-2\rr+3,
\end{equation}
\begin{equation}\label{5.8}
\begin{split}
A_3(\rr,\z)=\,&(64 \rr^3-576 \rr^2+1712
   \rr-1680) \z^6+ (48 \rr^2-276 \rr+393) \z^5 \cr
& + (-192 \rr^3+1392 \rr^2-3276 \rr+2487)\z^4 +(-96 \rr^2+408 \rr-408)\z^3 \cr
&+ (192 \rr^3-1056 \rr^2+1848 \rr-1032) \z^2 + (48 \rr^2-132 \rr+87) \z\cr
&-64 \rr^3 +240 \rr^2 -284 \rr+105,
\end{split}
\end{equation}
\begin{equation}\label{5.9}
\begin{split}
A_5(\rr,\z)=\,
& (2592 \rr^5-38880 \rr^4+231840 \rr^3-686880 \rr^2+1011008
   \rr-591360) \z^{10}\cr
&+(2160 \rr^4-25200 \rr^3+109560
   \rr^2-210350 \rr+150465) \z^9\cr
&+(-12960 \rr^5+170640 \rr^4-889200 \rr^3+2290440\rr^2-2913490 \rr+1462335)\z^8\cr
&+ (-8640 \rr^4+86400 \rr^3-319440
   \rr^2+516600 \rr-307610) \z^7\cr
&+(25920 \rr^5-293760 \rr^4+1310400\rr^3-2872560 \rr^2+3091080 \rr-1305190) \z^6\cr
&+ (12960 \rr^4-108000 \rr^3+329760
   \rr^2-436500 \rr+211240) \z^5\cr
&+(-25920 \rr^5+246240\rr^4-914400 \rr^3+1657440 \rr^2-1465580 \rr+505560) \z^4\cr
&+ (-8640 \rr^4+57600 \rr^3-139440
   \rr^2+145400 \rr-55110) \z^3\cr
&+(12960\rr^5-99360 \rr^4+295200 \rr^3-424560 \rr^2+295240 \rr-79290)\z^2 \cr
&+ (2160 \rr^4-10800 \rr^3+19560 \rr^2-15150
   \rr+4215) \z\cr
&-2592 \rr^5+15120 \rr^4-33840 \rr^3+36120 \rr^2 -18258\rr+3465.
\end{split}
\end{equation}
Note that $A_1(r,\z),\frac 13A_3(r,\z),\frac 15A_5(r,\z)\in\Bbb Z[\z]$. By assumption, $z\ne -1$. Hence $z$ is a common root of $A_1(r,\z)$, $\frac 13A_3(r,\z)$ and $\frac 15A_5(r,\z)$ in $\f_{q^2}$.

For $i,j\in\{1,3,5\}$, $i<j$, let $R_{ij}$ denote the resultant of $\frac 1iA_i(r,\z)$ and $\frac 1jA_j(r,\z)$ treated as polynomials in $\Bbb Z[\z]$. With computer assistance, we find that
\begin{align}
\label{5.10}
R_{1,3}&=-\frac{2^9}{3^2}(r-3)^2(4r-7)^2h_{1,3}(r),\\
\label{5.11}
R_{1,5}&=-\frac{2^{13}}{5^2}(r-3)^2(2r-3)h_{1,5}(r),\\
\label{5.12}
R_{3,5}&=-\frac{2^{43}}{3^8\cdot 5^6}(r-3)^2h_{3,5}(r),
\end{align}
where $h_{1,3}, h_{1,5}, h_{3,5}\in\Bbb Z[\rr]$ are given below:
\begin{equation}\label{5.13}
h_{1,3}=8 \rr^2-32 \rr+23,
\end{equation}
\begin{equation}\label{5.14}
\begin{split}
h_{1,5}=\,&417792 \rr^7-6220800 \rr^6+38904064 \rr^5-132226368 \rr^4+263268784 \rr^3\cr
&-306413232 \rr^2+192510160 \rr-50177175,
\end{split}
\end{equation}
\begin{equation}\label{5.15}
\begin{split}
h_{3,5}=\,&
21119053438918950050070528 \rr^{28}\cr
&-1217802457851859262370742272
   \rr^{27} \cr & +33695682531771885297793499136
   \rr^{26} \cr & -595640449348053724576692043776
   \rr^{25} \cr & +7555969260252555507865718095872
   \rr^{24} \cr & -73248462876723300946488091213824
   \rr^{23} \cr & +564228757279539380358831153348608
   \rr^{22} \cr & -3545224885397742156794256357851136
   \rr^{21} \cr & +18509319909682455299397692929605632
   \rr^{20} \cr & -81379894636486495421006534236176384
   \rr^{19} \cr & +304297772497625768143155803975057408
   \rr^{18} \cr & -974678059666820944936074552648400896
   \rr^{17} \cr & +2688005876401609920053983363863560192
   \rr^{16} \cr & -6404564332483459115509239563149737984
   \rr^{15} \cr & +13209244119542504384062885644435258368
   \rr^{14}\cr
&-23596172317266885038522124141212199936\rr^{13} \cr 
& +36479664536657352839953925385556203392\rr^{12}\cr
&-48706132767092416541853122916769684224\rr^{11} \cr 
& +55959309692846308509760642138106884928\rr^{10}\cr 
&-55030900064677544182509145667872622016\rr^9 \cr 
& +45980684429130187438483339370188443584\rr^8\cr 
&-32316001910874766059468388718091396312\rr^7 \cr 
& +18846211417895804224301626730504310302\rr^6\cr
&-8951174935307409932529759009356097240\rr^5 \cr 
& +3372192212650154034800139553730275800\rr^4\cr
&-968910653712017064601924894849677750\rr^3 \cr 
& +199340494276328696648165448026683125\rr^2\cr 
&-26137880501033434757380449712031250\rr \cr 
& +1640196174434693231689160015671875.
\end{split}
\end{equation}
By assumption, $(r-3)(4r-7)\not\equiv 0\pmod p$. For the moment, also assume that $2r-3\not\equiv 0\pmod p$. Then $r$ is a common root of $h_{1,3}$, $h_{1,5}$ and $h_{3,5}$ (in $\f_p$). Hence $\text{Res}_{\f_p[\rr]}(h_{1,3},h_{1,5})=\text{Res}_{\f_p[\rr]}(h_{1,3},h_{3,5})=0$, where $\text{Res}_{\f_p[\rr]}(\cdot\,,\,\cdot)$ denotes the resultant of two polynomials in $\f_p[\rr]$. On the other hand, direct computation (with computer assistance gives
\begin{equation}\label{5.16}
\text{Res}_{\Bbb Z[\rr]}(h_{1,3},h_{1,5})=2^{20}\cdot 3^{4}\cdot 23\cdot 8681
\end{equation}
and
\begin{equation}\label{5.17}
\text{Res}_{\Bbb Z[\rr]}(h_{1,3},h_{3,5})=2^{65}\cdot 3^{18}\cdot 7\cdot 41\cdot
185871968716987252172951795997086716801
\end{equation}
in prime factorization. Since $p\ne 2,3$, $\text{Res}_{\Bbb Z[\rr]}(h_{1,3},h_{1,5})$ and $\text{Res}_{\Bbb Z[\rr]}(h_{1,3},h_{3,5})$ cannot be both $0$ in $\f_p$, which is a contradiction.

Now assume $r\equiv 3/2\pmod p$. By \eqref{5.7} -- \eqref{5.9},
\[
A_1\Bigl(\frac 32,\,\z\Bigr)=-\z(3\z-1),
\]
\[
A_3\Bigl(\frac 32,\,\z\Bigr)=-3 (64 \z^6-29 \z^5-19 \z^4+4 \z^3-4 \z^2+\z-1),
\]
\[
A_5\Bigl(\frac 32,\,\z\Bigr)=-5 \z(3003 \z^9-1467 \z^8-1998 \z^7+718 \z^6-88 \z^5+88 \z^4-18 \z^3+18\z^2-3 \z+3).
\]
Recall that $z$ is a common root of the above polynomials. Thus $z=1/3$. However,
\[
A_3\Bigl(\frac 32,\,\frac 13\Bigr)=\frac{2^7\cdot 7}{3^5},\qquad A_5\Bigl(\frac 32,\,\frac 13\Bigr)=-\frac{2^{11}\cdot 5\cdot 13}{3^9},
\]
which cannot be both $0$ (in $\f_p$). So we have a contradiction.

\section{Proof of Theorem~\ref{T2.1} with $r\equiv 3\pmod p$}

In addition to the common conditions in Theorem~\ref{T2.1}, we assume that $q\ge r^2-4r+5$ and $r\equiv 3\pmod p$. Assume to the contrary that $f=f_{q,r,2,a}$ is a PP of $\f_{q^2}$. 

By \eqref{5.7}, $A_1(r,\z)=\z-3$, and hence $z=3$. In this case, $z=3$ is a common root of $A_\alpha(r,\z)$ for $\alpha=1,3,5$, so \eqref{5.7} -- \eqref{5.9} produce no contradiction. In fact, for $r\equiv 3\pmod p$ and $z=3$, no contradiction can be derived from \eqref{3.9} with $\alpha<p$. A special value of $\alpha$ needs to be considered.

Write $r=kp^l+3$, where $k,l>0$, $p\nmid k$. Choose $\alpha=p^l$. Then in \eqref{3.3},
\[
c=\left\lceil\frac{(r-2)\alpha+r}{q+1}\right\rceil\le\left\lceil\frac{(r-2)(r-3)+r}{q+1}\right\rceil=\left\lceil\frac{r^2-4r+6}{q+1}\right\rceil\le 1.
\]
So $c=1$ and
\begin{equation}\label{6.1}
\frac d2=\frac{q+1}2+\alpha-r\frac{\alpha+1}2=\frac{q+1}2+p^l-r\frac{p^l+1}2.
\end{equation}
It is clear from \eqref{6.1} that $d<q-1$. Since
\[
\binom\alpha i=
\begin{cases}
1&\text{if $i=0$ or $p^l$},\cr
0&\text{otherwise},
\end{cases}
\]
\eqref{3.9} gives
\[
0=\binom{\frac d2}{p^l}+\binom{\frac d2+\frac 12}{p^l}z-\binom{p^l+\frac d2}{p^l}z^{2p^l}-\binom{p^l+\frac d2+\frac 12}{p^l}z^{2p^l+1}.
\]
By \eqref{6.1}, the above equation becomes 
\begin{equation}\label{6.2}
\binom{u_1-r\frac{p^l+1}2}{p^l}+\binom{u_2-r\frac{p^l+1}2}{p^l}z-\binom{u_3-r\frac{p^l+1}2}{p^l}z^{2p^l}-\binom{u_4-r\frac{p^l+1}2}{p^l}z^{2p^l+1}=0,
\end{equation}
where $u_1=\frac 12+p^l$, $u_2=1+p^l$, $u_3=\frac12+2p^l$, $u_4=1+2p^l$. Since $z=3$, by \eqref{3.11}, we know that \eqref{6.2} also holds with $r$ replaced by $3$, that is, 
\begin{equation}\label{6.3}
\binom{u_1-3\frac{p^l+1}2}{p^l}+\binom{u_2-3\frac{p^l+1}2}{p^l}z-\binom{u_3-3\frac{p^l+1}2}{p^l}z^{2p^l}-\binom{u_4-3\frac{p^l+1}2}{p^l}z^{2p^l+1}=0.
\end{equation}
For $1\le i\le 4$, since
\[
u_i-3\frac{p^l+1}2-\Bigl(u_i-r\frac{p^l+1}2\Bigr)=kp^l\frac{p^l+1}2\equiv\frac k2p^l\pmod{p^{l+1}},
\]
we have
\[
\binom{u_i-3\frac{p^l+1}2}{p^l}-\binom{u_i-r\frac{p^l+1}2}{p^l}\equiv \frac k2\pmod p.
\]
Thus, subtracting \eqref{6.2} from \eqref{6.3} gives
\[
0=\frac k2\bigl(1+z-z^{2p^l}-z^{2p^l+1}\bigr)=\frac k2(1+z)^{p^l+1}(1-z)^{p^l}.
\]
Therefore $z=\pm 1$, which is a contradiction.

\section{Proof of Theorem~\ref{T2.1} with $r\not\equiv 3\pmod p$\\ and Either $p=3$ or $r\equiv 7/4\pmod p$}

We assume that $q\ge 8r-15$, $r\not\equiv 3\pmod p$, and either $p=3$ or $r\equiv 7/4\pmod p$. Again, assume to the contrary that $f=f_{q,r,2,a}$ is a PP of $\f_{q^2}$.

\subsection{The case $p=3$}\label{s7.1}\

In this case, \eqref{5.7} gives $A_1(r,\z)=-r\z^2+\z+r$. Recall that $\frac 13A_3(r,\z)\in\Bbb Z[\z]$. Write $r\equiv r_0\pmod{3^2}$, where $1\le r_0\le 8$, $3\nmid r_0$. Then $\frac 13A_3(r,\z)\equiv\frac 13 A_3(r_0,\z)\pmod 3$. Therefore,
\[
\text{Res}_{\f_3[\z]}\Bigl(A_1(r,\z),\,\frac 13A_3(r,\z)\Bigr)=\text{Res}_{\f_3[\z]}\Bigl(A_1(r_0,\z),\,\frac 13A_3(r_0,\z)\Bigr),
\]
which is obtained by setting $r=r_0$ in \eqref{5.10}. It turns out that
\[
\text{Res}_{\f_3[\z]}\Bigl(A_1(r,\z),\,\frac 13A_3(r,\z)\Bigr)=
\begin{cases}
0&\text{if}\ r_0=4,\cr
1&\text{if}\ r_0=5,8,\cr
-1&\text{if}\ r_0=1,2,7.
\end{cases}
\]
It is easy to verify that $A_1(4,\z) =-\z^2+\z+1$ divides both $\frac 13A_3(4,\z)$ and $A_5(4,\z)$ in $\f_3[\z]$. (In fact, $\frac 13A_3(4,\z)=\z(\z^2-\z-1)(\z^3-\z^2-\z-1)$ and $\frac 15A_5(4,\z)=\z^2(\z+1)^2(\z^2-\z-1)$.) This means that the consideration of the sums $\Theta(\alpha)$ in \eqref{5.4} -- \eqref{5.6} with $\alpha=1,3,5$ produces no contradiction.

To extract additional information, we consider \eqref{3.9} with $\alpha=7$. Clearly, $q\ge 8r-15\ge 3^2$. In \eqref{3.3}, we have 
\begin{equation}\label{7.1}
c=\left\lceil\frac{(\alpha+1)r-2\alpha}{q+1}\right\rceil=\left\lceil\frac{8r-14}{q+1}\right\rceil
\end{equation}
and 
\begin{equation}\label{7.2}
\frac d2=c\frac{q+1}2+\alpha-r\frac{\alpha+1}2\equiv\frac c2+7-4\cdot\frac{7+1}2=\frac c2\pmod{3^2}.
\end{equation}
Since $q\ge 8r-15$, \eqref{7.1} gives $c=1$. Note that $d\ne q-1$. Now \eqref{3.9} with $\alpha=7$ gives
\begin{equation}\label{7.3}
\Theta(7):=\sum_i\binom 7i(-1)^i\biggl[\binom{i+\frac c2}7z^{2i}+\binom{i+\frac{c+1}2}7z^{2i+1}\biggr]=0.
\end{equation}
On the other hand, direct computation shows that
\begin{equation}\label{7.4}
\Theta(7):=z^6(z+1)A_7(z),
\end{equation}
where
\[
A_7(\z)=(\z^3-\z^2-\z-1)(\z^5-\z^2+\z+1).
\]
We must have $A_7(z)=0=A_1(4,z)=0$ in $\f_3$. However,
\[
\text{gcd}_{\f_3[\z]}\bigl(A_1(4,\z),A_7(\z)\bigr)=1,
\]
which is a contradiction.

\begin{rmk}\label{R7.1}\rm 
If $q<8r-15$, then for $\alpha=7$, we have $c=2$. In this case, \eqref{7.4} becomes
\begin{equation}\label{7.5}
\Theta(7)=z^3 (z+1)(z^2-z-1) (z^8+z^5-z^3+z^2+z+1).
\end{equation}
Since $A_1(4,z)=-z^2+z+1$ appears in \eqref{7.5} as a factor, no contradiction is reached.
\end{rmk}

\subsection{The case $r\equiv 7/4\pmod p$}\

We may assume that $p>3$ because of Subsection~\ref{s7.1}. If $p=5$, then $r\equiv 3\pmod p$, which is false. Therefore, we assume that $p>5$.

When $r\equiv 7/4\pmod p$, \eqref{5.7} -- \eqref{5.9} in $\f_p[\z]$ becomes
\begin{align}
\label{7.6}
A_1(r,\z)\,&=\frac{1}{2} (-5 \z^2+2 \z-1),\\
\label{7.7}
A_3(r,\z)\,&=
-3 \z (5 \z^2-2 \z+1)(7 \z^3-\z^2-\z-1),\\
\label{7.8}
A_5(r,\z)\,&=-\frac 5{2^5}B_5(\z),
\end{align}
where
\begin{equation}\label{7.9}
\begin{split}
B_5(\z)=\,& 33649 \z^{10}-19726 \z^9-2219 \z^8-1096 \z^7-1214 \z^6 \cr
&+44 \z^5-814 \z^4+184\z^3-499 \z^2+114 \z-231.
\end{split}
\end{equation}
We find that 
\[
\text{Res}_{\Bbb Z[\z]}\bigl(-5\z^2+2\z-1,\, B_5(\z)\bigr)=2^{27}\cdot 3^2\cdot 181.
\]
Thus we must have $p=181$. In $\f_{181}[\z]$, \eqref{7.6} -- \eqref{7.8} become
\begin{align*}
A_1(r,\z)=\,&88 (\z+116) (\z+137),\\
A_3(r,\z)=\,&76 \z (\z+116) (\z+137) (\z+159) (\z^2+177 \z+67),\\
A_5(r,\z)=\,&178 (\z+116) (\z+142)\\
&(\z^8+8 \z^7+163 \z^6+69 \z^5+68 \z^4+165 \z^3+62
   \z^2+33 \z+152),
\end{align*}
where the factors in the above are all irreducible in $\f_{181}[\z]$. Since $z$ is a common root of $A_1(r,\z)$, $A_3(r,\z)$ and $A_5(r,\z)$, we must have $\z=-116=65$.

As in Subsection~\ref{s7.1}, again we need additional information from \eqref{3.9} with $\alpha=7$. Since $q\ge 8r-15$, in \eqref{3.3}, we have $c=1$ and 
\begin{equation}\label{7.10}
\frac d2\equiv \frac c2+7-\frac 74\cdot\frac 82=\frac 12\pmod q.
\end{equation}
Note that $d\ne q-1$ and $q\ge 8r-15>7$. Thus \eqref{3.9} (with $\alpha=7$) and \eqref{7.10} imply that 
\[
\Theta(7):=\sum_i\binom 7i(-1)^i\biggl[\binom{i+\frac 12}7 65^{2i}+\binom{i+1}7 65^{2i+1}\biggr]=0
\]
in $\f_{181}$. However, direct computation gives
\[
\Theta(7)=46\ne 0,
\]
which is a contradiction.

The proof of Theorem~\ref{T2.1} is now complete.

\end{document}